\documentclass[11pt]{article}
\usepackage{amsmath}
\usepackage{amssymb}
\usepackage{amsthm}

\def\fpd#1#2{{\displaystyle\frac{\partial #1}{\partial #2}}}

\def\vf#1{\frac{\partial}{\partial #1}}

\def\conn#1#2#3{\setbox1=\hbox{$\scriptstyle{#2}{#3}$}%
\setbox2=\hbox to\wd1{$\hfil\scriptstyle{#1}\hfil$}
\Gamma^{\!\box2}_{\!\box1}}

\def\D{\mathfrak{D}}

\def\l{\mathfrak{l}}
\def\L{\mathfrak{L}}
\def\poly{\mathfrak{p}}

\def\R{\mathbb{R}}
\def\r{\mathbf{r}}
\def\S{\mathcal{S}}

\def\bxi{\boldsymbol{\xi}}
\def\boeta{\boldsymbol{\eta}}
\def\B{\mathbf{B}}
\def\e{\mathbf{e}}
\def\K{\mathbf{K}}
\def\y{\mathbf{y}}

\def\onehalf{{\textstyle\frac12}}

\def\threehalf{{\textstyle\frac32}}

\def\im{\mathop{\mathrm{im}}}
\def\id{\mathop{\mathrm{id}}}

\newcommand{\dT}{d_{\mathrm{T}}}
\newcommand{\eval}[2]{\left. #1 \right|_{#2}}
\newcommand{\g}{\gamma}
\newcommand{\G}{\Gamma}
\newcommand{\gt}{\tilde{\gamma}}
\newcommand{\Gt}{\tilde{\G}}
\newcommand{\iT}{i_{\mathrm{T}}}
\newcommand{\jt}{\tilde{\jmath}}
\newcommand{\Lie}{\mathcal{L}}
\newcommand{\tr}{\mbox{\textsc{t}}}
\newcommand{\ve}{\varepsilon}
\newcommand{\vt}{\vartheta}

\renewcommand{\dddot}[1]{%
\setbox0=\hbox{$#1$}%
\dimen0=-\wd0%
\dimen1=\dimen0%
\dot{#1}%
\advance\dimen0 by -0.2em%
\hspace{\dimen0}%
\dot{\phantom{#1}}%
\advance\dimen1 by 0.4em%
\hspace{\dimen1}%
\dot{\phantom{#1}}%
\hspace{-0.2em}%
}

\renewcommand{\ddddot}[1]{%
\setbox0=\hbox{$#1$}%
\dimen0=-\wd0%
\dimen1=\dimen0%
#1%
\advance\dimen0 by -0.2em%
\hspace{\dimen0}%
\ddot{\phantom{#1}}%
\advance\dimen1 by 0.4em%
\hspace{\dimen1}%
\ddot{\phantom{#1}}%
\hspace{-0.2em}%
}

\newtheorem{thm}{Theorem}
\newtheorem{lemma}{Lemma}
\newtheorem{cor}{Corollary}
\newtheorem{prop}{Proposition}
\newtheorem{defn}{Definition}

\newcommand{\art}[6]{#1: #2 {\it #3\/} {\bf #4} (#5) #6}

\newcommand{\book}[4]{#1: {\it #2\/} (#3, #4)}
\newcommand{\artinbook}[7]{#1: #2. In:\ {\it #3\/} ed #4 (#5, #6) #7}
\newcommand{\arxiv}[3]{#1: #2. Preprint:\ arXiv:#3}

\begin{document}

\title{Homogeneity and projective equivalence of differential 
equation fields}

\author{M.\ Crampin\\
Department of Mathematics,
Ghent University\\
Krijgslaan 281, B--9000 Gent, Belgium\\
and D.\,J.\ Saunders\\
Department of Mathematics, Faculty of Science,\\
The University of Ostrava\\
30.\ dubna 22, 701 03 Ostrava, Czech Republic}

\maketitle

\begin{abstract}\noindent We propose definitions of homogeneity and
projective equivalence for systems of ordinary differential equations
of order greater than two, which allow us to generalize the concept 
of a spray (for systems of order two). We show that the 
Euler-Lagrange fields of parametric Lagrangians of order greater than 
one which are regular (in a natural sense that we define) form a 
projective equivalence class of homogeneous systems. We show further 
that the geodesics, or base integral curves, of projectively 
equivalent homogeneous differential equation fields are the same apart from
orientation-preserving reparametrization; that is, homogeneous 
differential equation fields determine systems of paths.
\end{abstract}

\subsubsection*{MSC}
34A26, 70H03, 70H50.

\subsubsection*{Keywords}
Ordinary differential equations; homogeneity; projective equivalence; 
parametric Lagrangian; reparametrization.

\section{Introduction}
This paper is concerned with systems of ordinary differential 
equations 
\[
y^i_{n+1} = \G^i(y^i,y^i_1, \ldots, y^i_n);\quad 
y^i_r=\frac{d^ry^i}{dx^r} 
\]
of order $n+1$, $n\geq 1$, in $m$ dependent variables $y^i$,
$i=1,2,\ldots,m$ and one independent variable $x$.  Second-order
systems of this kind ($n=1$) have been studied intensively using
methods of differential geometry, by taking advantage of the fact that
the equations are those satisfied by the integral curves of a vector
field of a certain type, a so-called second-order differential
equation field (often abbreviated to SODE) on the tangent bundle $TM$
of an $m$-dimensional differential manifold $M$, or some open
submanifold of it.  Higher-order systems may be studied by analoguous
methods, where now the differential equation vector field in question
lives (in the case of equations of order $n+1$) on the bundle $T^nM$
of $n$-velocities, in other words $n$-jets at zero of curves in $M$
(whose domains contain zero), or some open submanifold of it.  This is
the approach adopted here.  (For a recent account of the geometric
theory of differential equations, covering both second- and
higher-order systems, and containing many references, see
\cite{buc}.  However, this reference is mainly concerned with aspects
of the theory not covered here, namely generalized connections and
related matters.)

In the theory of second-order differential equation fields an
important role is played by a class of fields which are homogeneous in
a certain sense.  Let $T_\circ M$ be the slit tangent bundle of $M$
($TM$ with the zero section deleted).  Let $\Delta$ be the Liouville
vector field, that is, the infinitesimal generator of dilations of the
fibres of $T_\circ M$.  A second-order differential equation field
$\Gamma$ on $T_\circ M$ is called a spray if it satisfies
$[\Delta,\Gamma]=\Gamma$.  The canonical spray of a Finsler space,
whose integral curves project onto geodesics with constant speed
parametrization, is a helpful example.  If $\Gamma$ is a spray then
the distribution $\D^+$ spanned by $\Gamma$ and $\Delta$ is
involutive.  Any second-order differential equation field $\Gamma'$
contained in $\D^+$ is said to be projectively equivalent to $\Gamma$,
and its base integral curves are obtained from those of $\Gamma$ by
reparametrization.  Slightly more generally, for the purposes of this
paper we shall say that a second-order differential equation field
$\Gamma$ is homogeneous if the distribution $\D^+$ spanned by $\Gamma$
and $\Delta$ is involutive.  It is not difficult to show that if
$\Gamma$ is homogeneous in this sense then there is a projectively
equivalent second-order differential equation field which is a spray. 
In Finsler geometry the canonical spray is the Euler-Lagrange field 
of the energy, which is a regular Lagrangian. But instead of the 
energy one may consider the Finsler function itself as a Lagrangian.
Because of its assumed homogeneity it is parametric:\ that is to say, 
the corresponding action integral is invariant under 
reparametrizations. Such a parametric Lagrangian determines, not a 
single Euler-Lagrange field, but a projective equivalence class of 
second-order differential equation fields, which in the Finslerian 
case is the projective class of the canonical spray.

Our purpose in the present paper is to generalize the concept of
homogeneity so that it applies to differential equation fields of
higher order.  For this purpose we introduce the jet group of order
$n$, $L^n$, which is defined as follows.  Consider the local
diffeomorphisms $\phi$ of $\R$, defined in a neighbourhood of zero and
satisfying $\phi(0)=0$.  Then $L^n$ is the set of $n$-jets at zero of
such local diffeomorphisms; it is a group under composition, and even
a Lie group.  Moreover, it acts to the right on $T^nM$.  Let $\D$ be
the distribution spanned by the generators of the action.  We say that
a differential equation field $\Gamma$ is homogeneous if the
distribution $\D^+$ spanned by $\Gamma$ and $\D$ is involutive.  This
generalizes the concept described above for second-order differential
equation fields because when $n=1$, $\D$ is just the one-dimensional
distribution spanned by $\Delta$.  (A preliminary attempt to examine
homogeneity of higher-order differential equation fields is to be
found in \cite{hom}; however, this paper is again concerend mainly
with connection theory.)

Section 3 below is devoted to the definition of homogeneity and its
immediate consequences.
 
With this definition we show that the other properties of homgeneous
second-order differential equation fields discussed above also
generalize.  The extremals of a Lagrangian of order $n$ satisfy
differential equations of order $2n$.  If the Lagrangian is
parametric, so that it satisfies the so-called Zermelo conditions,
then its extremals are determined only up to reparametrization.  We
show in Section 4 that, provided it is sufficiently regular, a
parametric Lagrangian determines a projective equivalence class of
homogeneous differential equation fields.  This result gives us the
reassurence that our definition is not vacuous.  

In Section 5 we show that the base integral curves of projectively
equivalent homogeneous differential equation fields differ only in
parametrization.

Section 6 of the paper contains some examples, while Section 2 
is devoted to the requisite geometry of jet manifolds and jet groups.

We shall use the Einstein summation convention for coordinate indices
such as $i$; other sums will be indicated explicitly.

\section{Geometrical background}

We consider differentiable manifolds of class $C^\infty$ which are
Hausdorff, second-countable and (unless otherwise specified)
connected.

Let $M$ be such a manifold with $\dim M = m$ and with local
coordinates $(y^i)$.  Consider curves $\g : (a,b) \to M$, where
$0\in(a,b)$, and let $T^n M$ be the set of $n$-jets at zero
$\{j^n_0\g\}$, with coordinates $(y^i_r)$ for $0 \le r \le n$ given by
\[
y^i_r(j^n_0\g) = \eval{\frac{d^r \g^i}{dx^r}}{0}  .
\]
It is a standard result that $T^n M$ is a manifold with $\dim T^n M =
(n+1)m$, and that the maps $\tau_n : T^n M \to M$ and
$\tau_{n,n^\prime} : T^n M \to T^{n^\prime}M$ ($n^\prime < n$) given
by
\[
\tau_n(j^n_0\g) = \g(0)  ,
\qquad \tau_{n,n^\prime}(j^n_0\g) = j^{n^\prime}_0\g
\]
are fibre bundles.

The manifold $T^{n+1}M$ has a natural identification with a
submanifold of $TT^n M$, obtained because the map $t \mapsto j^n_t\g$
defines a curve
\[
\jt^n\g : t \mapsto j^n_0(\g\circ\tr_t)
\]
in $T^n M$, where $\tr_t : \R \to \R$ is the translation $x \mapsto
x+t$.  We therefore obtain the inclusion
\[
T^{n+1}M \to TT^n M  , \qquad j^{n+1}_0\g \mapsto j_0 \jt^n\g
\]
given in coordinates by $y^i_{r+1} \mapsto \dot{y}^i_r$.  We may think
of this inclusion as a section of the pull-back vector bundle
$\tau_{n+1,n}^* TT^n M \to T^{n+1}M$, in other words as a vector field
along the projection $\tau_{n+1,n}$; with such an interpretation it is
the total derivative
\[
\dT = \sum_{r=0}^n y^i_{r+1} \vf{y^i_r}  .
\] 
The projection $\tau_{n+1,n} : T^{n+1}M \to T^n M$ is in fact an
affine bundle, modelled on the vector bundle $V\tau_{n,n-1} \to T^n M$
of very vertical tangent vectors on $T^n M$.  The affine action is
just addition in the fibres of $TT^n M \to T^n M$:
\[
(y^i_r; 0, \ldots, 0, z^i) + 
(y^i_r; \dot{y}^i_0, \ldots, \dot{y}^i_{n-1}, \dot{y}^i_n)
\;\mapsto\; (y^i_r; \dot{y}^i_0, \ldots, \dot{y}^i_{n-1}, z^i + \dot{y}^i_n)
\]
gives rise, using the identification $\dot{y}^i_r = y^i_{r+1}$ on
$T^{n+1}M$, to
\[
(y^i_r; z^i) + (y^i_r; y^i_{n+1}) \;\mapsto\; (y^i_r; z^i + y^i_{n+1})  .
\]

We shall make use later of the vertical endomorphism on $T^nM$.  This
is the type $(1,1)$ tensor field $S$ given in coordinates by
\[
S=\sum_{r=1}^n r\vf{y^i_r}\otimes dy^i_{r-1};
\]
it is canonically defined, and generalizes the well-known vertical 
endomorphism or tangent structure on $TM$. It is not difficult to 
show that for any 1-form $\alpha$ on $T^nM$, 
\[ 
S(\dT\alpha)-\dT(S\alpha)=\tau_{n+1,n}^*\alpha
\]
(where in the first instance $S$ is the vertical endomorphism on
$T^{n+1}M$, in the second the one on $T^nM$).  This may conveniently
be written $S\dT-\dT S=1$, with the pull-back map understood (note
that we assume action on 1-forms; there is a more complicated formula
for action on forms of higher degree).

A section $\Gt : T^n M \to T^{n+1}M$ of the affine bundle
$\tau_{n+1,n}$ is called a \emph{differential equation field}; in
coordinates it is indeed a differential equation
\[
y^i_{n+1} = \G^i(y^i_0, \ldots, y^i_n)  .
\]
A curve $\g : (a,b) \to M$ is a geodesic of the equation if
$\jt^{n+1}\g = \Gt \circ \jt^n\g$.  Composing $\Gt$ with the inclusion
$T^{n+1}M \to TT^n M$ gives a vector field $\G$ on $T^n M$ of the
particular form
\[
\G = \sum_{r=0}^{n-1} y^i_{r+1} \vf{y^i_r} + \G^i \vf{y^i_n}  ; 
\]
every integral curve $\gt : (a,b) \to T^n M$ of $\G$ is of the form
$\gt = \jt^n\g$ for some geodesic $\g$ of $\Gt$; that is to say, the
geodesics of $\Gt$ are the base integral curves of $\G$.  We use the
term `differential equation field' to refer to $\G$ as well as $\Gt$.

Now consider local diffeomorphisms $\phi$ of $\R$ defined in a
neighbourhood of zero and satisfying $\phi(0)=0$.  Let $L^n$ be the
set of $n$-jets at zero $\{j^n_0\phi\}$ of these local
diffeomorphisms; this is a group under composition,
\[
j^n_0\phi_1 \cdot j^n_0\phi_2 = j^n_0(\phi_1 \circ \phi_2)
\]
and is a Lie group. As a manifold it has two connected components.
The component of the identity is a subgroup $L^{n+}$ of index 2
containing jets of local diffeomorphisms satisfying
$\phi^\prime(0)>0$.  The map
\[
j^n\phi_0 \mapsto \bigl( \phi^\prime(0), \phi^{\prime\prime}(0), \ldots, 
\phi^{(n)}(0) \bigr) \in \R^n
\]
is a global coordinate system on $L^n$ (and $L^{n+}$).  We can obtain
an explicit formula for the product in these coordinates by using an
expression for the $n$th derivative of a composition of functions.
This is Fa\`{a} di Bruno's formula~\cite{FdB}, which we take in the
form involving the Bell polynomials:
\[
(\xi \circ \eta)^{(n)}(x) =
\sum_{r=1}^n \xi^{(r)}(\eta(x))B_n^r\left( \eta^\prime(x), 
\eta^{\prime\prime}(x), 
\ldots, \eta^{(n+1-r)}(x) \right)  ;
\]
here $\xi$ and $\eta$ are functions of some variable $t$ with (in our
case) $\xi(0)=\eta(0)=0$, and $B_n^r(\eta_1, \eta_2, \ldots,
\eta_{n+1-r})$ is a polynomial in the $n + 1 - r$ variables $\eta_p$,
for which the following explicit formula is known:
\begin{multline*}
B_n^r(\eta_1, \eta_2, \ldots, \eta_{n+1-r}) \\
= \sum \frac{n!}{q_1! q_2!\cdots q_{n+1-r}!} 
\left( \frac{\eta_1}{1!} \right)^{q_1} 
\left( \frac{\eta_2}{2!} \right)^{q_2}
\ldots \left( \frac{\eta_{n+1-r}}{(n + 1 - r)!} \right)^{q_{n+1-r}}
\end{multline*}
where the sum is over all non-negative integers $q_1, q_2, \ldots,
q_{n+1-r}$ such that $q_1+q_2+ \ldots + q_{n+1-r} = r$ and $q_1+2q_2+
\ldots +(n+1-r)q_{n+1-r} = n$.  (This formula may easily be derived by
taking $\xi(x)=x^r$ and expressing $\eta(x)$ as a formal power series in
$x$:
\[
\eta(x)=\frac{\eta_1}{1!}x+\frac{\eta_2}{2!}x^2+\frac{\eta_3}{3!}x^3+\cdots;
\]
then $(r!/n!)B_n^r(\eta_1, \eta_2, \ldots, \eta_{n+1-r})$ is the
coefficient of $x^n$ in the formal power series for $(\eta(x))^r$.  In
principle $B_n^r$ could depend on $\eta_p$ with $p>n+1-r$, but in 
practice, since $\eta(x)$ contains no constant term there can be no 
contribution to $x^n$ in $(\eta(x))^r$ coming from the term 
$\eta_px^p/p!$ if $p>n+1-r$.) To express multiplication in 
$L^n$ in terms of the global coordinate system introduced above, a 
vectorial notation is convenient. We denote by $\bxi$, $\boeta$ 
elements of $\R^n$, considered as row vectors, and by $\B(\boeta)$ 
the (upper triangular) matrix whose $p,q$ entry is 
$B_q^p(\eta_1,\eta_2,\ldots)$. Then the product in $L^n$ is given by
\[
\bxi\cdot\boeta=\bxi\B(\boeta);
\]
we must of course assume that $\xi_1$ and $\eta_1$ are non-zero. Note 
that the matrix which determines multiplication in $L^{n'}$ with $n'<n$ 
is just the $n'\times n'$ submatrix of $\B(\boeta)$ in the upper left 
corner. For instance, taking $n=4$ gives
\[
\B(\boeta)=
\left(
\begin{array}{cccc}
\eta_1&\eta_2&\eta_3&\eta_4\\
0&\eta_1^2&3\eta_1\eta_2&4\eta_1\eta_3+3\eta_2^2\\
0&0&\eta_1^3&6\eta_1^2\eta_2\\
0&0&0&\eta_1^4
\end{array}
\right).
\]

There are obvious projections $\lambda_{n,1} : L^n \to L^1$ and
$\lambda_{n,1}^+ : L^{n+} \to L^{1+}$ given by $j^n_0\phi \mapsto
j^1_0\phi$; they are homomorphisms because
\[
\lambda_{n,1}(j^n_0\phi_1 \cdot j^n_0\phi_2) = 
\lambda_{n,1}(j^n_0(\phi_1 \circ \phi_2)) 
= j^1_0(\phi_1 \circ \phi_2) = j^1_0 \phi_1 \cdot j^1_0 \phi_2  .
\]
The kernel of $\lambda_{n,1}$ is therefore a normal subgroup $K^n \lhd
L^n$.  By inspection of the matrix $\B$ above in the case $n=4$ we see
that $K^2$ and $K^3$ are abelian but $K^4$ is not; indeed $K^n$ is
non-abelian whenever $n \ge 4$.

We may, in addition, identify $L^1$ with a subgroup of $L^n$ by mapping
$j^1_0\phi$ to $j^n_0 \mu_{\phi^\prime(0)}$, where $\mu_s : \R \to \R$
is the multiplication diffeomorphism $\mu_s(t) = st$ for $s \ne 0$.
Thus $L^n$ may be regarded as a semidirect product $L^1 \rtimes K^n$,
using the action of $L^1$ on $K^n$ by conjugation.

Furthermore, we may write $L^n = L^2 \rtimes J^n$, where $J^n$ is the
kernel of the homomorphism $\lambda_{n,2} : L^n \to L^2$ given by
$j^n_0\phi \mapsto j^2_0\phi$.  For the inclusion of $L^2$ in $L^n$ we
take, as a representative of $j^2_0\phi$, not the quadratic polynomial
$\phi(x) = ax + \frac{1}{2}bx^2$ but the M\"{o}bius
transformation
\[
\phi(x) = \frac{ax}{1 - a^{-1}bx}
\]
because the composition of two of these transformations is another map
of the same form.  By calculating the derivatives of $\phi$ one finds
a polynomial representative of $j^n_0\phi$,
\[
\sum_{r=1}^n \frac{b^{r-1}x^r }{a^{r-2}}
\]
(see~\cite{KMS}).  Note that the kernels of the homomorphisms
$\lambda_{n,r}$ do not give semidirect product decompositions when
$r>2$.

We next obtain a basis for the Lie algebra $\l^n$ of the jet group 
$L^n$ in terms of the coordinates defined above. For this purpose we 
need the following lemma.

\begin{lemma}\label{bell}
For given $n$, $p=1,2,\ldots,n$, and $r=1,2,\ldots,n+1-p$, 
\[
\eval{\fpd{B_n^p}{\eta_r}}{(1,0,\ldots,0)}=0
\]
unless $p=n+1-r$, when 
\[
\eval{\fpd{B_n^p}{\eta_r}}{(1,0,\ldots,0)}=\frac{n!}{(n-r)!r!}.
\]
\end{lemma}

\begin{proof}
We consider the cases $r=1$, $r\geq 2$ separately. We have
\begin{align*}
\fpd{B_n^p}{\eta_1}&=\sum\frac{n!}{(q_1-1)!q_2!\cdots q_{n+1-p}!1!}\\
&\qquad\times
\left(\frac{\eta_1}{1!}\right)^{q_1-1}\left(\frac{\eta_2}{2!}\right)^{q_2}
\ldots\left(\frac{\eta_{n+1-r}}{(n + 1 - r)!}\right)^{q_{n+1-r}}
\end{align*}
where the sum is over all non-negative integers $q_1, q_2, \ldots,
q_{n+1-r}$ with $q_1>0$ such that $q_1+q_2+ \ldots + q_{n+1-r} = p$
and $q_1+2q_2+ \ldots +(n+1-r)q_{n+1-r} = n$. With $\eta_1=1$, 
$\eta_r=0$ for $r\geq 2$ we get a non-zero contribution to the sum 
only if $q_r=0$ for $r\geq 2$, when we must have $q_1=p$ and $q_1=n$. 
That is, there is no non-zero term in the sum unless $p=n$, and so
\[
\eval{\fpd{B_n^p}{\eta_1}}{(1,0,\ldots,0)}=
\left\{
\begin{array}{ll}
0&p\neq n\\
\displaystyle{\frac{n!}{(n-1)!1!}}&p=n
\end{array}
\right..
\]
For $r\geq 2$ on the other hand
\begin{align*}
\fpd{B_n^p}{\eta_r}&=\sum\frac{n!}{q_1!\cdots(q_r-1)!\cdots q_{n+1-p}!r!}\\
&\qquad\times
\left(\frac{\eta_1}{1!}\right)^{q_1}\ldots
\left(\frac{\eta_r}{r!}\right)^{q_r-1}
\ldots\left(\frac{\eta_{n+1-r}}{(n + 1 - r)!}\right)^{q_{n+1-r}}
\end{align*}
where the sum is now over all non-negative integers $q_1, q_2, \ldots,
q_{n+1-r}$ with $q_r>0$ such that $q_1+q_2+ \ldots + q_{n+1-r} = p$
and $q_1+2q_2+ \ldots +(n+1-r)q_{n+1-r} = n$.  With $\eta_1=1$,
$\eta_s=0$ for $s\geq 2$ we get a non-zero contribution to the sum only
if $q_r=1$, $q_s=0$ for $s\geq 2$, $s\neq r$.  We must therefore 
have $q_1+1=p$ and $q_1+r=n$, that is, $p=n+1-r$ and $q_1=n-r$.  Then
\[
\eval{\fpd{B_n^p}{\xi_r}}{(1,0,\ldots,0)}=
\left\{
\begin{array}{ll}
0&p\neq n+1-r\\
\displaystyle{\frac{n!}{(n-r)!r!}}&p=n+1-r
\end{array}
\right..
\qedhere
\]
\end{proof}

We denote by $\delta^r$, $r=1,2,\ldots,n$, the left-invariant vector
field on $L^n$ which takes the value
\[
r!\vf{y_r}
\]
at the identity. Here $(y_1,y_2,\ldots,y_n)$ are the coordinates on 
$L^n$. The factor $r!$ is included for later convenience. We now 
obtain an explicit expression for $\delta^r$ in terms of coordinates.

\begin{prop}
\[
\delta^r=\sum_{s=r}^n\frac{s!}{(s-r)!}y_{s+1-r}\vf{y_s}.
\]
\end{prop}

\begin{proof}
Let $\y=(y_1,y_2,\ldots,y_n)$ be a generic point of $L^n$. The 
identity $\e$ has coordinates $(1,0,\ldots,0)$. We have
\[
\delta^r_{\y}=L_{\y}|_{*\e}\delta^r_{\e}
\]
where $L_\y$ is left multiplication by $\y$:\ $L_\y\boeta=\y\B(\boeta)$. 
Then
\begin{align*}
L_{\y}|_{*\e}\vf{y^r}&=
\sum_{s=1}^n\sum_{p=1}^sy_p\eval{\fpd{B_s^p}{\eta_r}}{(1,0,\ldots,0)}\vf{y_s}\\
&=\sum_{s=r}^n\frac{s!}{(s-r)!r!}y_{s+1-r}\vf{y_s}.
\qedhere
\end{align*}
\end{proof}

\begin{cor}
\[
[\delta^r,\delta^s]=
\left\{
\begin{array}{ll}
(r-s)\delta^{r+s-1}&\mbox{if $r+s\leq n+1$}\\
0&\mbox{otherwise}
\end{array}
\right..
\]\qed
\end{cor}
\noindent One reason for introducing the factor $r!$ in the definition of 
$\delta^r$ is to simplify the expression for the bracket.

We can now make some observations about the Lie algebra $\l^n$ which
correspond to properties of the Lie group $L^n$ mentioned earlier.  In
the first place, $\{\delta^1,\delta^2\}$ span a subalgebra.
Furthermore, for any $p=1,2,\ldots,n$, $\langle\delta^r:r\geq
p\rangle$ is an ideal in $\l^n$ (since for any $r\geq p$ and $s\geq
1$, $r+s-1\geq p$).  Let $\l_p^n\subset \l^n$ be the ideal just
defined (so that in particular $\l^n=\l_1^n$).  Then for $p>1$,
$\l^n/\l_p^n\sim\l^{p-1}$.

The ideals $\l^n_p$ for $p>1$ are nilpotent.  It is known that for a
nilpotent Lie algebra, exponentiation is a surjective map onto the
simply connected Lie group of which it is the algebra; it follows in
particular that exponentiation maps $\l^n_2$ onto $K^n$ (see
\cite{KMS}).  But it is possible to prove this directly, as we now show.
The Lie algebra $\l^n$ consists of linear vector fields on $\R^n$, and
so exponentiation in the group $L^n$ coincides with matrix
exponentiation.  In the case of an element
$\kappa=\sum_{r=2}^nk_r\delta^r$ of $\l^n_2$ the matrix $\K_n$ to be
exponentiated is strictly lower triangular:\ its elements are given by
\[
(\K_n)_{rs}=
\left\{
\begin{array}{ll}
0&r<s\\
\displaystyle{\frac{r!}{(s-1)!}}k_{r+1-s}&r\geq s
\end{array}
\right..
\]
Notice that the $(n-1)\times(n-1)$ submatrix in the upper left corner
is just the matrix $\K_{n-1}$ corresponding to the element
$\sum_{r=2}^{n-1}k_r\delta^r$ of $\l^{n-1}_2$.  As an example, with
$n=5$ we have
\renewcommand{\arraystretch}{1.2}
\[
\K_5=\left(
\begin{array}{ccccc}
0&0&0&0&0\\
\frac{2!}{0!}k_2&0&0&0&0\\
\frac{3!}{0!}k_3&\frac{3!}{1!}k_2&0&0&0\\
\frac{4!}{0!}k_4&\frac{4!}{1!}k_3&\frac{4!}{2!}k_2&0&0\\
\frac{5!}{0!}k_5&\frac{5!}{1!}k_4&\frac{5!}{2!}k_3&\frac{5!}{3!}k_2&0
\end{array}
\right).
\]
\renewcommand{\arraystretch}{1}

\begin{prop}
For any $\y=(1,y_2,y_3,\ldots,y_n)^T\in K^n$ there is $\kappa\in\l^n_2$ 
such that $\y=\exp\kappa$.
\end{prop}

\begin{proof}
As a column vector in $\R^n$, $\exp\kappa\in L^n$ is just
$(\exp\K_n)\e$, where $\e=(1,0,\ldots,0)^T$; that is to say,
$\exp\kappa$ is the left column of $\exp\K_n$.  It is easy to see that
the first $n-1$ entries in the left column of $\exp\K_n$ comprise the
left column of the $(n-1)\times(n-1)$ matrix $\exp\K_{n-1}$, while the
lower left corner of $\exp\K_n$ is of the form
$n!k_n+p(k_2,k_3,\ldots,k_{n-1})$ where $p$ is a polynomial in the
indicated variables.  Since $\K_n$ is strictly lower triangular,
$\exp\K_n$ has 1s down the diagonal.  We have to show that for any
$y_2,y_3,\ldots,y_n$ we can choose $k_2,k_3,\ldots,k_n$ such that the
left column of $\exp\K_n$ is $(1,y_2,y_3,\ldots,y_n)^T$.  We proceed
by induction.  Thus we assume that we can find
$k_2,k_3,\ldots,k_{n-1}$ such that the left column of $\exp\K_{n-1}$
is $(1,y_2,y_3,\ldots,y_{n-1})^T$.  With these values of
$k_2,k_3,\ldots,k_{n-1}$ we set
$k_n=(y_n-p(k_2,k_3,\ldots,k_{n-1}))/n!$: then the left column of
$\exp\K_n$ is $(1,y_2,y_3,\ldots,y_n)^T$ as required.  We have
\[
\exp\K_2=\left(
\begin{array}{ll}
1&0\\
2!k_2&1
\end{array}
\right),
\]
so the first step certainly works.
\end{proof} 

We next turn to the relationship between $\l^n$ and certain vector
fields on $\R$.  An element $\delta$ of $\l^n$ determines a
one-parameter subgroup of $L^n$, namely $t\mapsto\exp(t\delta)$:\
observe first that the integral curve of $\delta$ through $\y\in L^n$
is $t\mapsto \y\cdot\exp(t\delta)$.  Now let $\phi_t$ be a
one-parameter group of diffeomorphisms of $\R$ such that
$\phi_t(0)=0$.  Then $\tilde{\phi}_t=j^n_0\phi_t$ is evidently a
one-parameter subgroup of $L^n$ whose infinitesimal generator, a
vector field $\chi$ on $L^n$, is left invariant and therefore an
element of $\l^n$.  For any $\y\in L^n$, $\chi_{\y}$ is the
tangent vector at $t=0$ to the curve $t\mapsto \y\cdot\tilde{\phi}_t$.
But the $s$th component of $\y\cdot\tilde{\phi}_t$ is given by
\[
(\y\cdot\tilde{\phi}_t)_s
=\sum_{p=1}^s y_p B_s^p\left(\eval{\fpd{\phi_t}{x}}{0},
\eval{\frac{\partial^2\phi_t}{\partial x^2}}{0},\ldots,
\eval{\frac{\partial^{s+1-p}\phi_t}{\partial x^{s+p-1}}}{0}\right)y_p.
\]
Note that since $\phi_0(x)=x$, the argument of $B_s^p$ at $t=0$ is
$(1,0,\ldots,0)$, while
\[
\frac{\partial}{\partial t}
\left(\frac{\partial^r\phi_t}{\partial x^r}\right)_{(0,0)}
=\frac{\partial^r}{\partial x^r}
\left(\frac{\partial\phi_t}{\partial t}\right)_{(0,0)}
=\eval{\frac{d^rX}{dx^r}}{0}
\]
where $X(x)\partial/\partial x$ is the infinitesimal generator of 
$\phi_t$ (a vector field on $\R$). On differentiating the formula for 
$(\y\cdot\tilde{\phi}_t)_s$ with respect to $t$ at $t=0$ we obtain
\[
\chi=\sum_{s=1}^n\left(
\sum_{p=1}^s\frac{s!}{(s-p)!p!}\eval{\frac{d^pX}{dx^p}}{0}y_{s+1-p}\right)\vf{y_s}
=\sum_{p=1}^n\frac{1}{p!}\eval{\frac{d^pX}{dx^p}}{0}\delta^p.
\]
In fact 
\[
\Phi : X\mapsto
\sum_{p=1}^n\frac{1}{p!}\eval{\frac{d^pX}{dx^p}}{0}\delta^p
\]
is a linear map from the space of vector fields on $\R$ vanishing at
the origin onto $\l^n$, such that $\Phi(X\partial/\partial x)$ depends
only on $j^n_0X$.  In particular, $\Phi(x^r\partial/\partial
x)=\delta^r$ for $r=1,2,\ldots,n$ while $\Phi(x^r\partial/\partial
x)=0$ for $r>n$.  Now
\[
\left[x^r\vf{x},x^s\vf{x}\right]=(s-r)x^{r+s-1}\vf{x},
\]
so $\Phi$ is an anti-homomorphism.

Let $\poly$ be the Lie algebra of vector fields on $\R$ whose
coefficients are formal power series in $x$, and let $\poly^n$ be the
subalgebra of those vector fields which vanish to order $n$ at 0, that
is, whose coefficient begins with $x^{n+1}$.  Then $\poly^0$ is the
subalgebra of formal power series vector fields which vanish at 0; and
for $n>0$, $\poly^n$ is an ideal in $\poly^0$.  If we think of $\Phi$
as a map $\poly^0\to\l^n$ it is a surjective anti-homomorphism with
kernel $\poly^n$, and therefore defines an anti-isomorphism
$\poly^0/\poly^n\to\l^n$.  To put things another way, we can realise
$\l^n$ as the space of vector fields on $\R$ whose coefficients are
polynomials of order $n$ which vanish at the origin, with bracket the
negative of the ordinary bracket of vector fields followed by
truncation at order $n$.  This alternative realisation is found
elsewhere in the literature.  The fact that the bracket is related to
the negative of the ordinary bracket is not too surprising when one
recalls that if the diffeomorphism group of $\R$ is regarded as an
infinite dimensional Lie group, its Lie algebra is the space of vector
fields on $\R$ with compact support, but with bracket the negative of
the ordinary vector field bracket.

The group $L^n$ has a right action $\alpha_n$ on $T^n M$ given by
composition of jets,
\[
\alpha_n : L^n \times T^n M \to T^n M  , 
\qquad (j^n_0\phi, j^n_0\g) \mapsto j^n_0(\g\circ\phi)  ;
\]
the action is fibred over the identity on $M$.  It restricts to the
`regular' submanifold $T^n_\circ M$ (that is, of $n$-jets $j^n_0\g$
where $\g$ is a curve with $\g^\prime(0) \ne 0$, so that $\g$ is an
immersion near zero) because composing a diffeomorphism with an
immersion gives another immersion.  We may use Fa\'{a}~di~Bruno's
formula recursively to see that the restricted action is free, using
the coordinates $(y^i_r)$, $1 \le i \le n$, defined on complete fibres
over $M$:\ by regularity at least one coordinate $y^i_1$ must be
non-zero at any given point, so if $\alpha_n(j^n_0\phi, j^n_0\g) =
j^n_0\g$ we see successively that $\phi^\prime(0) = 1$ and then that
$\phi^{\prime\prime}(0) = 0$, $\phi^{\prime\prime\prime}(0) = 0$,
\ldots, so that $j^n_0\phi = 1_{L^n}$.  The orbit space of $T^n_\circ
M$ under the action $\alpha_n$ has a manifold structure (and is,
indeed, a Hausdorff manifold); we shall denote it by $PT^n M$.  The
subgroup $L^{n+}$ acts in the same way, and its orbit space will be
denoted by $P^+ T^n M$; this is a double cover of $PT^n M$.  Let
$\rho_n : T^n_\circ M \to PT^n M$ and $\rho^+_n : T^n_\circ M \to P^+
T^n M$ be the projections.

In the present work we shall in effect be interested in the
circumstances when a differential equation field on $T^n M$ (more
accurately, on $T^n_\circ M$) `passes to the quotient' to determine a
line-element field on $PT^n M$ or an oriented line-element field on
$P^+ T^n M$.  In view of the identification of $T^{n+1}M$ with an
affine sub-bundle of $TT^n M$ over $T^n M$, it is of some interest to
consider the identification of $PT^{n+1}$ with a submanifold of
$PT(PT^n M)$.  One may do this, for instance, by using the action of
the tangent group $TL^n$ on $TT^n_\circ M$ to give $T(PT^n M)$, and
then the action of $L^1$ on the open submanifold $T_\circ(PT^n M)$ to
give $PT(PT^n M)$.  In fact it is known that $PT^{n+1}$ is an affine
bundle over $PT^n M$, even though $PT(PT^n M) \to PT^n M$ certainly
does not have an affine structure, as its fibres are compact (they are
projective spaces):\ see, for instance, the discussion in~\cite{KL},
which uses the fact that the kernel of the homomorphism $L^{n+1} \to
L^n$ is abelian.  Similar arguments hold in the oriented case.

We shall need to know the fundamental vector fields on $T^nM$ 
of the action of $L^n$. We denote by $\Delta^r$ the fundamental 
vector field corresponding to $\delta^r\in\l^n$. Then $\Delta^r$ is 
the infinitesimal generator of the one-parameter group 
$R_{\exp(t\delta^r)}$ (where $R$ denotes the right action). In terms 
of the coordinates $(y^i_r)$ on $T^nM$, the action of $\boeta\in L^n$ 
is given by
\[
(R_{\boeta}(y^j_s))^i_r=\sum_{p=1}^ry^i_pB_r^p(\eta_1,\eta_2,\ldots,\eta_{r+1-p}).
\]
By a by now familiar type of argument invoking Lemma~\ref{bell} we 
obtain
\begin{prop}
\[
\Delta^r=\sum_{s=r}^n\frac{s!}{(s-r)!}y^i_{s+1-r}\vf{y^i_s}.
\]
\end{prop}
\noindent In particular $\Delta^1$, which corresponds to the
infinitesimal generator of dilations of $\R$, is given by
\[
\Delta^1=y^i_1\vf{y^i_1}+2y^i_2\vf{y^i_2}+\cdots+ny^i_n\vf{y^i_n}.
\]

\begin{cor}
\[
[\Delta^r,\Delta^s]=
\left\{
\begin{array}{ll}
(r-s)\Delta^{r+s-1}&\mbox{if $r+s\leq n+1$}\\
0&\mbox{otherwise}
\end{array}
\right.. 
\]
\end{cor}

\begin{cor}
In terms of the vertical endomorphism $S$
\begin{enumerate}
\item a vector field $\Gamma$ on $T^n_\circ M$ is a differential 
equation field if and only if $S(\Gamma)=\Delta^1$;
\item $\Delta^{r+1}=S(\Delta^r)=S^r(\Delta^1)=S^{r+1}(\Gamma)$.
\end{enumerate}
\end{cor}

We shall not use the results of Corollary 3 directly,
but we mention them because they provide an alternative starting point
for the geometrical analysis of differential equation fields which may
be found elsewhere in the literature. 

We shall however use in Section 4 a result which is related to item 1
of Corollary 3, which we now explain.  We denote by $i$ the operator
of interior product of a vector field with a form, so that for example
for a 1-form $\alpha$, $i_X\alpha=\alpha(X)$.  We extend this notation
to apply to the total derivative:\ if $\alpha$ is a $1$-form on
$T^n_\circ M$ then $\iT\alpha$ is a function on $T^{n+1}_\circ M$.
Then $\iT(S\alpha)=i_{\Delta^1}(\tau_{n+1,n}^*\alpha)$, which may
conveniently be written $\iT S=i_{\Delta^1}$, with the pull-back map
understood.  Note that we assert only that this formula holds for
action on 1-forms (there is a more complicated formula for action on
forms of higher degree).  Similarly, in relation to item 2, we have
(again for action on 1-forms) $\iT S^r=i_{\Delta^r}$.

\section{Homogeneous higher-order systems}

We now consider differential equation (vector) fields on $T^n_\circ M$.

\begin{defn} 
A differential equation field $\Gamma$ on $T^n_\circ M$ is
{\em homogeneous\/} if the distribution $\D^+$ spanned by $\Gamma$ and
the $\Delta^r$ is involutive.
\end{defn}

\begin{prop}
The differential equation field $\Gamma$ is homogeneous if and only 
if there are functions $\lambda^r$, $r=1,2,\ldots,n$ such that 
\[
[\Delta^1,\Gamma]=\Gamma+\lambda^1\Delta^n,\qquad 
[\Delta^r,\Gamma]=r\Delta^{r-1}+\lambda^r\Delta^n,
\quad r=2,3,\ldots, n.
\]
When they exist, such functions must satisfy the following consistency 
conditions, where $1<r,s\leq n$, $r\neq s$:
\begin{align*}
&\Delta^1(\lambda^r)-\Delta^r(\lambda^1)=(n+1-r)\lambda^r;\\
&\Delta^r(\lambda^s)-\Delta^s(\lambda^r)=(r-s)\lambda^{r+s-1},\quad
r+s\leq n+1;\\
&\Delta^r(\lambda^s)-\Delta^s(\lambda^r)=-(n+1)(r-s),\quad
r+s= n+2;\\
&\Delta^r(\lambda^s)-\Delta^s(\lambda^r)=0,\quad
r+s>n+2.
\end{align*}
\end{prop}

\begin{proof}
A straightforward coordinate calculation shows that
$[\Delta^1,\Gamma]-\Gamma$ and $[\Delta^r,\Gamma]-r\Delta^{r-1}$ are 
very vertical. Thus in order for $\D^+$ to be involutive there 
must be functions $\lambda^r$ such that 
\[
[\Delta^1,\Gamma]=\Gamma+\lambda^1\Delta^n,\qquad 
[\Delta^r,\Gamma]=r\Delta^{r-1}+\lambda^r\Delta^n,
\quad r=2,3,\ldots, n.
\]
The consistency conditions follow from the Jacobi identities.
\end{proof}

In terms of coordinates the conditions for 
\[
\G = \sum_{r=0}^{n-1} y^i_{r+1} \vf{y^i_r} + \G^i \vf{y^i_n}  
\]
to be homogeneous are
\[
\Delta^1(\G^i)=(n+1)\G^i+n!\lambda^1y^i_1,\qquad 
\Delta^r(\G^i)=\frac{(n+1)!}{(n+1-r)!}y^i_{n+2-r}+n!\lambda^ry^i_1. 
\]

If $\Gamma$ is a differential equation field such that
$\D^+$ is involutive, so is $\Gamma+\mu\Delta^n$ for any function 
$\mu$ on $T^n_\circ M$, and of course it belongs to the same involutive 
distribution.

One cannot, in general, demand that for $n\geq 3$ a differential
equation field $\Gamma$ satisfies the strong conditions
\[
[\Delta^1,\Gamma]=\Gamma,\qquad [\Delta^r,\Gamma]=r\Delta^{r-1},
\quad r=2,3,\ldots, n;
\]
the Jacobi identities would be inconsistent with the bracket relations
for the $\Delta^r$.  (This point is discussed more fully in
\cite{hom}, where examples of equations which do satisfy the
conditions for $n=2$ are given.)  Recall that $\{\Delta^1,\Delta^2\}$
span a subalgebra of $\D$; it is easily verified that the conditions
\[
[\Delta^1,\Gamma]=\Gamma,\qquad [\Delta^2,\Gamma]=2\Delta^1
\]
are consistent. We shall now show that for any homogeneous 
$\hat{\Gamma}$ it is possible to find $\mu$ such that if 
$\Gamma=\hat{\Gamma}+\mu\Delta^n$ then $\Gamma$ satisfies the 
conditions above (as well as 
$[\Delta^r,\Gamma]=r\Delta^{r-1}+\lambda^r\Delta^n$ for $r>2$); in 
other words, if $\Gamma$ is homogeneous, without loss of generality we 
may assume that $\lambda^1=\lambda^2=0$. We need a couple of lemmas.

\begin{lemma}\label{subm} 
There is a submanifold $\S$ of $T^n_\circ M$ of codimension 1 such that
$\Delta^1$ is transverse to $\S$ and $\Delta^2$ is tangent to it.
There is a submanifold $\S'$ of $\S$ of codimension 1 such that
$\Delta^2|_{\S}$ is transverse to $\S'$.
\end{lemma}
\begin{proof}
Let $g$ be any Riemannian metric on $M$.  Then
$g_{ij}y^i_1y^j_1$ is a well-defined function on $T^n M$
(because the $y^i_1$ transform as the components of tangent vectors),
and $g_{ij}y^i_1y^j_1>0$ on $T^n_\circ M$.  Let 
$\varphi(y)=\sqrt{g_{ij}y^i_1y^j_1}$, and $\S=\{y\in
T^n_\circ M:\varphi(y)=1\}$.  Then $\Delta^1(\varphi)=\varphi$ and
$\Delta^2(\varphi)=0$, so $\Delta^1$ and $\Delta^2$ are respectively
transverse and tangent to $\S$.  Next, let
$\varphi'=\onehalf\Gamma(\varphi)$ for any differential
equation field $\Gamma$ on $T^n_\circ M$.  We have
\[
\Delta^2(\varphi')=\onehalf\Delta^2(\Gamma(\varphi))
=\onehalf\Gamma(\Delta^2(\varphi))+\Delta^1(\varphi)
=\varphi,
\]
since $[\Delta^2,\Gamma]-2\Delta^1$ is very vertical. Now 
\[
2\varphi\varphi'=\onehalf\fpd{g_{ij}}{y^k}y^i_1y^j_1y^k_1+g_{ij}y^i_1y^j_2
=g_{ij}y^i_1(y^j_2+\conn jkly^k_1y^l_1),
\]
where the $\conn kij$ are the connection coefficients of the
Levi-Civita connection of the metric $g$, from which it is clear that
$\S'=\{y\in\S:\varphi'(y)=0\}$ is a codimension 1 submanifold of $\S$,
and $\Delta^2|_{\S}$ is transverse to it.
\end{proof}
\begin{lemma}\label{diffe}
Let $X$ be a complete vector field on a manifold $M$, with 1-parameter
group $\phi_t$, and $\S$ a codimension 1 submanifold of $M$ transverse
to $X$ such that $\{\phi_t(\S):t\in\R\}=M$.  Let $f$ be any smooth
function on $M$, $k$ any constant, and $z_0$ any smooth function on
$\S$.  Then there is a unique smooth function $z$ on $M$ such that
$X(z)+kz=f$ and $z|_{\S}=z_0$.  In particular, if $f=0$ and $z_0=0$ then
$z=0$.
\end{lemma}
\begin{proof}
One can integrate the differential equation 
\[
\frac{dz}{dt}+kz=f(\phi_t(x)),
\]
where $x$ is any fixed point of $\S$, by the integrating factor method.
\end{proof}\noindent
(The restriction that $k$ should be constant covers the situation 
encountered below, but clearly more general equations could be 
considered.)

\begin{thm}\label{spray}
Let $\hat{\Gamma}$ be a homogeneous higher-order differential equation
field on $T^n_\circ M$, so that $\D^+$ is involutive.  Then there is a
higher-order differential equation field
$\Gamma=\hat{\Gamma}+\mu\Delta^n\in\D^+$ such that 
\[
[\Delta^1,\Gamma]=\Gamma,\quad [\Delta^2,\Gamma]=2\Delta^1.
\]
\end{thm}
\begin{proof}
Let $\S$ be a codimension 1 submanifold of $T^n_\circ M$ such that $\Delta^1$ is
transverse to $\S$ and $\Delta^2$ is tangent to it, and $\S'$ a 
codimension 1 submanifold of $\S$ such that $\Delta^2|_{\S}$ is 
transverse to $\S'$.

We may write
\[
[\Delta^1,\hat{\Gamma}]=\hat{\Gamma}+\hat{\lambda}^1\Delta^n,\qquad 
[\Delta^2,\hat{\Gamma}]=2\Delta^1+\hat{\lambda}^2\Delta^n.
\]
Then with $\Gamma=\hat{\Gamma}+\mu\Delta^n$,
\[
[\Delta^1,\Gamma]=\Gamma+(\Delta^1(\mu)+(2-n)\mu+\hat{\lambda}^1)\Delta^n.
\]
By Lemma \ref{diffe} we can solve the equation
$\Delta^1(\mu)+(2-n)\mu=-\hat{\lambda}^1$, assigning the value of $\mu$ on
$\S$ arbitrarily.  Also,
\[
[\Delta^2,\Gamma]=2\Delta^1+(\Delta^2(\mu)+\hat{\lambda}^2)\Delta^n.
\]
Now $\Delta^2$ is tangent to $\S$, and $\mu$ is undetermined on that
submanifold, so we may solve the equation
$\Delta^2(\mu)=-\hat{\lambda}^2$ there, assigning the value of $\mu$
on $\S'$ arbitrarily.  With $\mu$ satisfying these two equations, for
$\Gamma$ we have $\lambda^1=0$ everywhere on $T^n_\circ M$ and
$\lambda^2=0$ on $\S$.  But from the consistency condition for
$\lambda^1$ and $\lambda^2$ applied to $\Gamma$ we have
$\Delta^1(\lambda^2)=(n-1)\lambda^2$ on $T^n_\circ M$.  It follows from
Lemma~\ref{diffe} and the fact that $\lambda^2=0$ on $\S_2$ that
$\lambda^2=0$ everywhere.
\end{proof}

We suggest, somewhat tentatively, that homogeneous higher-order
differential equation fields such that $[\Delta^1,\Gamma]=\Gamma$ and
$[\Delta^2,\Gamma]=2\Delta^1$ should be considered as generalized
sprays.

Finally in this section we observe that a homogeneous differential
equation field on $T^n_\circ M$ does indeed pass
to the quotient under the action of $L^{n+}$.

\begin{thm}
Let $\G$ be a homogeneous differential equation field on $T^n_\circ
M$.  Then the involutive distribution $\D^+$ defines an oriented 
line-element field on $P^+ T^n M$.
\end{thm}

\begin{proof}
The distribution $\D^+$ evidently projects to a line-element field on
$P^+ T^n M$.  To orient it we take, at any point
$\rho_n^+(j^n_0\gamma)$, the positive multiples of $\rho^+_{n*}
\Gamma_{j^n_0\gamma}$ for any representative point $j^n_0\gamma \in
T^n_\circ M$.
\end{proof}

\section{Parametric Lagrangians}
To show that the theory described above is not vacuous, we devote this
section to an important source of examples, the higher-order
differential equation fields obtained from parametric variational
problems.  (Related results, but for second-order variational problems
only, may be found in \cite{MM}.  Multiple-integral parametric
variational problems are discussed at length in \cite{CS2}, and indeed
the methods used here are related to those used in that paper, but of
course specialized to the case of a single integral.)

Let $L$ be a Lagrangian function, defined on $T^n_\circ M$.  The
\emph{Hilbert form} of $L$ is the 1-form
\[
\vt = \sum_{p=0}^{n-1} \frac{(-1)^p}{(p+1)!} \dT^p S^{p+1} dL
\]
defined on $T^{2n-1}_\circ M$.  This may be used to construct the
Euler-Lagrange form of the Lagrangian,
\[
\ve = dL - \dT \vt
\]
defined on $T^{2n}_\circ M$.  By construction, therefore,
\[
\ve = dL - \dT \sum_{p=0}^{n-1} \frac{(-1)^p}{(p+1)!} \dT^p S^{p+1} dL
= \sum_{p=0}^n \frac{(-1)^p}{p!} \dT^p S^p dL  .
\]
We may see that $\ve$ is horizontal over $M$, because $S\dT = \dT S + 1$
(modulo pullback), so that
\[
S\dT^p = \dT^p S + p \dT^{p-1}
\]
and therefore
\[
S\dT^p S^p = \dT^p S^{p+1} + p \dT^{p-1} S^p  .
\]
Thus $S\ve$ is a collapsing sum, and indeed
\[
S\ve = \sum_{p=0}^n \frac{(-1)^p}{p!} \dT^p S^{p+1} dL 
+ \sum_{p=1}^n \frac{p(-1)^p}{p!} \dT^{p-1} S^p dL
= \frac{(-1)^n}{n!} \dT^n S^{n+1} dL = 0
\]
as $L$ has order $n$.
\begin{lemma}
The Euler-Lagrange form $\ve$ vanishes along the extremals of any 
fixed-endpoint variational problem defined by $L$.
\end{lemma}
\begin{proof}
An extremal of such a problem is a curve $\g$ in $M$
such that
\[
\int_a^b (\jt^n \g)^* \Lie_{X^n} L \, dt = 0
\]
for any vector field $X$ on $M$ vanishing at the (fixed) image of the
endpoints of $[a,b]$, where $X^n$ denotes the prolongation of $X$ to
$T^n_\circ M$.  Then
\[
\int_a^b (\jt^n \g)^* \Lie_{X^n} L \, dt = \int_a^b (\jt^n \g)^* (i_{X^n} dL) \, dt 
+ \int_a^b (\jt^n \g)^* (di_{X^n} L) \, dt  ;
\]
but
\[
\int_a^b (\jt^n \g)^* (di_{X^n} L) \, dt = \int_a^b d((\jt^n \g)^* i_{X^n} L) \, dt 
= \bigl[ (\jt^n \g)^* i_{X^n} L \bigr]_a^b = 0
\]
as $X$ vanishes at the endpoints, so that
\[
\int_a^b (\jt^n \g)^* \Lie_{X^n} L \, dt = \int_a^b (\jt^n \g)^* (i_{X^n} dL) \, dt
= \int_a^b (\jt^{2n} \g)^* (i_{X^{2n}} (\ve + \dT \vt)) \, dt  .
\]
But now,
\begin{align*}
& \int_a^b (\jt^{2n} \g)^* (i_{X^{2n}} \dT \vt) \, dt 
= \int_a^b (\jt^{2n} \g)^* (\dT i_{X^{2n-1}} \vt) \, dt \\
& = \int_a^b d((\jt^{2n-1} \g)^* (i_{X^{2n-1}} \vt)) \, dt 
= \bigl[ (\jt^{2n-1} \g)^* (i_{X^{2n-1}} \vt) \bigr]_a^b = 0
\end{align*}
because vector field prolongations commute with total derivatives, and
the pull-back of $\dT$ to $\R$ is $d$.  Thus
\[
\int_a^b (\jt^n \g)^* \Lie_{X^n} L \, dt = \int_a^b (\jt^{2n} \g)^* (i_{X^{2n}} \ve) \, dt 
= \int_a^b (\jt^{2n} \g)^* (i_X \ve) \, dt
\]
because $\ve$ is horizontal over $M$.  As $X$ is arbitrary, it follows
that $(\jt^{2n} \g)^* \ve = 0$.
\end{proof}

Now suppose that $L$ is positively homogeneous; that is, that it
satisfies the Zermelo conditions
\[
\Delta^1(L) = L  , \qquad \Delta^r(L) = 0 \quad (r \ge 2)  .
\]
Such a Lagrangian is called a \emph{parametric Lagrangian}, and the
geodesics of the corresponding variational problem are invariant under
orientation-preserving reparametrizations~\cite{CS2}; in other words,
they are \emph{paths}.  We shall show that (subject to suitable
regularity conditions) they are also the geodesics of homogeneous
differential equation fields.  In order to investigate this, we first
establish a technical lemma.
\begin{lemma}
The Hilbert form $\vt$ and the Euler-Lagrange form $\ve$ satisfy
\[
\iT \vt = L  , \qquad \iT d\vt = - \ve
\]
(omitting the pull-back maps).
\end{lemma}
\begin{proof}
From $\iT \dT = \dT \iT$ and $\iT S^p = i_{\Delta^p}$ we obtain
\begin{align*}
\iT \dT^p S^{p+1} dL & = \dT^p \iT S^{p+1} dL 
= \dT^p i_{\Delta^{p+1}} dL \\
& =
\begin{cases}
i_{\Delta^1} dL = L & (p = 0) \\
0 & (p > 0)
\end{cases}
\end{align*}
so that
\[
\iT\vt = \sum_{p=0}^{n-1} \frac{(-1)^p}{(p+1)!} \iT\dT^p S^{p+1} dL = L  .
\]
In addition,
\[
\dT \vt = d \iT \vt + \iT d\vt = dL + \iT d\vt
\]
so that, from the definition of $\ve$,
\[
\ve = dL - \dT \vt = - \iT d\vt  .\qedhere
\]
\end{proof}
\begin{prop}
The characteristic distribution of the 2-form $d\vt$ satisfies
\[
\{ \G, \Delta^1, \Delta^2, \ldots, \Delta^{2n-1} \} \subset \ker d\vt
\]
whenever $\G$ is a differential equation field on $T^n_\circ M$ whose
geodesics are extremals of the variational problem with fixed
endpoints defined by $L$. 
\end{prop}
 
\begin{proof}
We remark first that $\vt$ satisfies $i_{\Delta^r}\vt = 0$ and $\Lie_{\Delta^r} \vt = 0$ for $1 \le r \le 2n-1$ (\cite{CS2}, Proposition~6.1 and Theorem~6.4). It follows  that
\[
i_{\Delta^r} d\vt = \Lie_{\Delta^r} \vt - d i_{\Delta^r} \vt = 0  ,
\]
and so $\{ \Delta^1, \Delta^2, \ldots, \Delta^{2n-1} \} \subset \ker d\vt$.

Now let $\Gt$ be a differential equation field on $T^{2n-1}_\circ M$, so that
\[
\Gt : T^{2n-1}_\circ M \to T^{2n}_\circ M  ,
\]
and let $\G : T^{2n-1}_\circ M \to i \bigl( T^{2n}_\circ M \bigr)
\subset TT^{2n-1}_\circ M$ be the corresponding vector field.

Suppose in particular that the integral curves of $\G$ (which must,
necessarily, be prolongations $\jt^{2n-1}\g$ of curves $\g$ in $M$) are
such that $\g$ is always an extremal of the variational problem.  Each
point of $T^{2n-1}_\circ M$ must lie on such an integral curve, and we
may suppose (by translation in the domain) that the point in question
is $\jt^{2n-1}\g(0)$, in other words that the point may be written as
$j^{2n-1}_0\g$.

The integral curve property means that $\Gt(j^{2n-1}_0\g) =
j^{2n}_0\g$, so that
\[
\G_{j^{2n-1}_0\g} = (i \circ \Gt)(j^{2n-1}_0\g) = i(j^{2n}_0\g)  ;
\]
thus for any vector $\xi\in T_{j^{2n-1}_0\g} T^{2n-1}_\circ M$ we have
\begin{align*}
d\vt_{j^{2n-1}_0\g} \bigl( \G_{j^{2n-1}_0\g}, \; \xi \bigr)
& = d\vt_{j^{2n-1}_0\g} \bigl( i(j^{2n}_0\g), \; \xi \bigr) \\
& = \bigl\langle (\iT d\vt)_{j^{2n}_0\g}, \; \xi \bigr\rangle \\
& = \bigl\langle -\ve_{j^{2n}_0\g}, \; \xi \bigr\rangle \\
& = 0
\end{align*}
because $\g$ is an extremal of $L$ so that $\ve_{j^{2n}_0\g} = 0$.
This calculation also uses the fact that the contraction $\iT$ is
just the inclusion map $i : T^{2n}_\circ M \to TT^{2n-1}_\circ M$ in
disguise.
\end{proof}
We say that the Lagrangian is \emph{regular} if the characteristic
distribution of $d\vt$ has dimension $2n$. Any differential equation
field in this distribution is an Euler-Lagrange field of the
variational problem.  Since the characteristic distribution of a
closed 2-form is involutive we have the following
theorem.

\begin{thm}
Any Euler-Lagrange field of a regular parametric Lagrangian is 
homogeneous.
\end{thm}

\section{Projective equivalence and reparametrization}

We now return to the study of general higher-order differential
equation fields.  

\begin{defn} Let $\Gamma$, $\Gamma'$ be homogeneous.
Then if $\Gamma'-\Gamma=\mu\Delta^n$ for some function $\mu$ we say
that $\Gamma$ and $\Gamma'$ are {\em projectively equivalent}.
\end{defn}

Projective equivalence is an equivalence relation.  The projective
equivalence class of a homogeneous differential equation field
consists of all the differential equation fields in the involutive
distribution $\D^+$.  We may summarize the result of
Theorem~\ref{spray} by saying that every projective equivalence class
of a homogeneous differential equation field contains a generalized
spray.

We have seen that the Euler-Lagrange fields of a regular parametric
Lagrangian consist of a projective equivalence class of homogeneous
differential equation fields.  So we should expect in this case that
the geodesics of projectively equivalent homogeneous differential
equation fields should define the same paths, that is, differ only in
parametrization.  In this section we shall show that this is true for
any projective equivalence class of homogeneous differential equation
fields, whether or not they come from a parametric Lagrangian.

In fact we shall prove two results about projective equivalence in
homogeneous systems:\ firstly, if $\Gamma'$ is projectively equivalent
to $\Gamma$ then a geodesic of $\Gamma'$ is a reparametrization of a
geodesic of $\Gamma$; secondly, the jet group acts on differential
equation fields in such a way as to map any homogeneous field to one
which is projectively equivalent to it.

\begin{thm}\label{projeq}
Let $\Gamma$ be a homogeneous differential equation field, and
$\Gamma'$ another which is projectively equivalent to it.  Let
$\gamma'$ be a geodesic of $\Gamma'$.  Then there is a geodesic
$\gamma$ of $\Gamma$ such that $\gamma'$ is obtained from $\gamma$ by
an orientation-preserving reparametrization.
\end{thm}

\begin{proof}
By assumption the distribution $\D^+$ is involutive, and thus
integrable.  Let $\L$ be a leaf of $\D^+$.  Then $\Gamma'$ is tangent
to $\L$, and any integral curve of $\Gamma'$ which meets $\L$ lies in
it.  If $\gamma'$ is a geodesic of $\Gamma'$ then $\jt^n\gamma'$ is an
integral curve of $\Gamma'$:\ let $\L$ be in fact the leaf of $\D^+$
containing (the image of) $\jt^n\gamma'$.  Since $\D\subset\D^+$, $\L$
is invariant under the action of $L^{n+}$, the identity component of
the jet group $L^n$, and consists of the orbits of points on
$\jt^n\gamma'$ under the action of $L^{n+}$.  Under projection
$\tau_n:T^n_\circ M\to M$, we have of course $\tau_n\circ
\jt^n\gamma'=\gamma'$, and for any $g\in L^{n+}$, $\tau_n\circ
g=\tau_n$.  Thus $\tau_n(\L)$ is the 1-dimensional submanifold of $M$
which is the oriented path of $\gamma'$.  Now $\Gamma$ also belongs to
$\D^+$, and so likewise any integral curve of $\Gamma$ which meets
$\L$ lies in it.  Take any point of $\L$, and the integral curve of
$\Gamma$ through that point:\ it is of the form $\jt^n\gamma$ where
$\gamma$ is a geodesic of $\Gamma$.  Then
$\gamma=\tau_n\circ\jt^n\gamma$ lies in $\tau_n(\L)$, that is to say,
$\gamma'$ and $\gamma$ have the same oriented path.  Since we are
restricted to $(y^i_1)\neq 0$, we must assume that $\gamma$ and
$\gamma'$ have nowhere vanishing tangent vectors.  It follows that one
is a reparametrization of the other, by a reparametrization that
preserves orientation.
\end{proof}

\begin{cor}\label{repara} 
Let $\Gamma$ be a homogeneous differential equation field of order
$n+1$.  Let $\gamma$ be the geodesic of $\Gamma$ with initial
conditions $(y^i,y^i_1,\ldots,y^i_n)$.  Let $(z^i_1,\ldots,z^i_n)$ be
the image of $(y^i_1,\ldots,y^i_n)$ under some element of $L^{n+}$ (so
that $(z^i_1)$ is a positive scalar multiple of $(y^i_1)$).  Then the
geodesic of $\Gamma$ with initial conditions
$(y^i,z^i_1,\ldots,z^i_n)$ is obtained from $\gamma$ by an
orientation-preserving reparametrization.
\end{cor}

The second result is rather more complicated to prove, and requires 
some preliminaries.

We shall need to distinguish notationally between the fundamental
vector fields of the actions of $L^n$ on $T^nM$ and $L^{n+1}$ on
$T^{n+1}M$.  Let us denote by $\Delta^r_n$ the vector field operating
on $T^n M$:\ that is to say
\[
\Delta^r_n=\sum_{s=r}^{n}\frac{s!}{(s-r)!}y^i_{s+1-r}\vf{y^i_{s}}.
\]
With a bit of licence we can write, for $r=1,2,\ldots,n$,
\[
\Delta^r_{n+1}=\Delta^r_n+\frac{(n+1)!}{(n+1-r)!}y^i_{n+2-r}\vf{y^i_{n+1}},
\]
while
\[
\Delta^{n+1}_{n+1}=(n+1)!y^i_1\vf{y^i_{n+1}}.
\]
In particular,
\[
\Delta^r_{n+1}(y^i_{n+1})=\frac{(n+1)!}{(n+1-r)!}y^i_{n+2-r},\quad
\Delta^{n+1}_{n+1}(y^i_{n+1})=(n+1)!y^i_1.
\]
Now the homogeneity conditions for an $(n+1)$th order differential
equation field $\Gamma$, when expressed in terms of the $\G^i$, are
\[
\Delta^1_n(\G^i)=(n+1)\G^i+n!\lambda^1y^i_1,\quad 
\Delta^r_n(\G^i)=\frac{(n+1)!}{(n+1-r)!}y^i_{n+2-r}+n!\lambda^ry^i_1. 
\]
Comparison with the formulae for $\Delta^r_{n+1}(y^i_{n+1})$ is 
suggestive.

Consider the section $\tilde{\Gamma}$ of $T_\circ^{n+1} M\to
T_\circ^n M$ corresponding to $\Gamma$, given by
$y^i_{n+1}=\G^i(y^j_0,y^j_1,\ldots,y^j_n)$.  We have, on
$\im\tilde{\Gamma}$, with $r=2,\ldots,n$,
\begin{align*}
&\Delta^1_{n+1}(y^i_{n+1}-\G^i)=(n+1)(y^i_{n+1}-\G^i)-n!\lambda^1y^i_1
=-n!\lambda^1y^i_1\\
&\Delta^r_{n+1}(y^i_{n+1}-\G^i)=-n!\lambda^ry^i_1\\
&\Delta^{n+1}_{n+1}(y^i_{n+1}-\G^i)=(n+1)!y^i_1.
\end{align*}
One way of restating this is that for $r=1,2,\ldots,n$ the vector
fields
\[
\Delta^r_{n+1}+\frac{\lambda^r}{(n+1)}\Delta^{n+1}_{n+1}
\]
(strictly speaking one should write $\tau^*_{n+1,n}\lambda^r$ rather
than just $\lambda^r$) are tangent to $\im\tilde{\Gamma}$.  Set
\[
\Delta^r_{n+1}+\frac{\lambda^r}{(n+1)}\Delta^{n+1}_{n+1}=\tilde{\Delta}^r.
\]
The $\tilde{\Delta}^r$ are well-defined vector fields on $T^{n+1} M$.

\begin{prop}
For $r,s=1,2,\ldots,n$
\[
[\tilde{\Delta}^r,\tilde{\Delta}^s]=\left\{
\begin{array}{ll}
(r-s)\tilde{\Delta}^{r+s-1}&\mbox{if $r+s\leq n+1$}\\
0&\mbox{otherwise}
\end{array}\right..
\]
\end{prop}
\begin{proof}
We have
\[
[\tilde{\Delta}^1,\tilde{\Delta}^r]=[\Delta^1_{n+1},\Delta^r_{n+1}]+
\frac{(\Delta^1_n(\lambda^r)-\Delta^r_n(\lambda^1)-n\lambda^r)}{(n+1)}\Delta^{n+1}_{n+1}
\]
while 
\[
[\tilde{\Delta}^r,\tilde{\Delta}^s]=[\Delta^r_{n+1},\Delta^s_{n+1}]+
\frac{(\Delta^r_n(\lambda^s)-\Delta^s_n(\lambda^r))}{(n+1)}\Delta^{n+1}_{n+1}
\]
(the $\lambda^r$ are functions on $T^n M$).  The stated results follow
easily, using the consistency conditions on the $\lambda^r$.  The one
case which is not entirely straightforward is $r+s=n+2$, when we have
\begin{align*}
[\tilde{\Delta}^r,\tilde{\Delta}^s]&=[\Delta^r_{n+1},\Delta^s_{n+1}]+
\frac{(\Delta^r_n(\lambda^s)-\Delta^s_n(\lambda^r))}{(n+1)}\Delta^{n+1}_{n+1}\\
&=(r-s)\Delta^{n+1}_{n+1}-\frac{(n+1)(r-s)}{(n+1)}\Delta^{n+1}_{n+1}=0,
\end{align*}
since in this case 
$\Delta^r(\lambda^s)-\Delta^s(\lambda^r)=-(n+1)(r-s)$.
\end{proof}
That is to say, when $\Gamma$ is homogeneous there is a representation of 
the algebra $\l^n$ on $\im\tilde{\Gamma}$.

Notice also that for $r\geq 2$, 
$[\Delta^{n+1}_{n+1},\tilde{\Delta}^r]=0$, since 
$[\Delta^{n+1}_{n+1},\Delta^r_{n+1}]=0$ and 
$\Delta^{n+1}_{n+1}(\lambda^r)=0$.

Let $\phi$ be an orientation-preserving local diffeomorphism of $\R$
leaving 0 fixed.  Let $\phi_n$ be the diffeomorphism of $T^nM$ given
by the action of $j^n_0\phi\in L^n$.  We have
$\tau_{n+1,n}\circ\phi_{n+1}=\phi_n\circ\tau_{n+1,n}$.  Let $\Gamma$
be an $(n+1)$st-order differential equation field,
$\tilde{\Gamma}:T^n_\circ M\to T^{n+1}_\circ M$ the corresponding
section.  Since $\tau_{n+1,n}\circ\tilde{\Gamma}=\id_{T^n_\circ M}$,
$\phi_{n+1}^{-1}\circ\tilde{\Gamma}\circ\phi_n$ is also a section of
$\tau_{n+1,n}$, and so defines a new $(n+1)$st-order differential
equation field, say $\Gamma_\phi$.  Suppose that $\Gamma$ is any
homogeneous $(n+1)$st-order differential equation field.  We shall
show that for any $\phi$, there is a function $\nu_\phi$ on $T^n_\circ
M$ such that
\[
\Gamma^i_\phi(y^j_r)=\Gamma^i(y^j_r)+\nu_\phi(y^j_r)y^i_1.
\]
Note that when this holds $\Gamma_\phi=\Gamma+\nu_\phi\Delta^n_n$, so
$\Gamma_\phi$ is projectively equivalent to $\Gamma$.  From the point
of view of representation by sections one might say that
$\tilde{\Gamma}_\phi$ is obtained by sliding $\tilde{\Gamma}$ along
the integral curves of $\Delta^{n+1}_{n+1}$.  To be more precise,
recalling that $T^{n+1}_\circ M$ is an affine bundle over $T^n_\circ
M$ modelled on the bundle of very vertical tangent vectors on
$T^n_\circ M$, $\tilde{\Gamma}_\phi$ is the translate of
$\tilde{\Gamma}$ by the very vertical vector field
$\nu_\phi\Delta^{n+1}_{n+1}$.
 
\begin{thm}  
Let $\Gamma$ be a homogeneous differential equation field and let
$\phi$ be an orientation-preserving local diffeomorphism of $\R$ 
leaving 0 fixed. Then $\Gamma_\phi$ is projectively equivalent to
$\Gamma$.
\end{thm}
\begin{proof}
Note first that
$\Gamma_{\phi_1\circ\phi_2}=(\Gamma_{\phi_1})_{\phi_2}$.  So if for
some $\phi_1,\phi_2$ and every homogeneous $\Gamma$, it is the case
that $\Gamma_{\phi_1}$ is projectively equivalent to $\Gamma$ and
$\Gamma_{\phi_2}$ is projectively equivalent to $\Gamma$ then
$\Gamma_{\phi_1\circ\phi_2}$ is projectively equivalent to $\Gamma$.
We pointed out in Section 2 that the finite order jet group $L^m$ is
the semi-direct product of $L^1$ and $K^m$, the subgroup consisting of
the jets of diffeomorphisms with derivative at 0 equal to the
identity, and further that every element of the latter is the
exponential of some element of its Lie algebra.  So the result about
$\Gamma_\phi$, or equivalently $\tilde{\Gamma}_\phi$, will be proved
if it can be proved when $\phi(x)=kx$ ($k$ a positive constant) and
when $\phi_{n+1}$ is the exponential of an element of
$\langle\Delta^r_{n+1}\rangle$ for $r=2,\ldots,n$.

Consider first the case where $\phi(x)=kx$.  Then
$\phi_n(y^i_r)=(k^ry^i_r)$.  We have
$\Delta^1_n(\G^i)=(n+1)\G^i+n!\lambda^1y^i_1$, which integrates to give
\[
e^{-(n+1)t}\G^i(e^{rt}y^j_r)-\G^i(y^j_r)=
n!\int_0^te^{-ns}\lambda^1(e^{rs}y^j_r)y^i_1ds.
\]
Then
\begin{align*}
\tilde{\Gamma}_\phi(y^j_r)&=\phi_{n+1}^{-1}\tilde{\Gamma}(\phi_n(y^j_r))
=\phi_{n+1}^{-1}(k^ry^i_r,\G^i(k^ry^j_r))\\
&=(y^i_r,k^{-(n+1)}\G^i(k^ry^j_r))\\
&=(y^i_r,\G^i(y^j_r)+\nu_\phi(y^j_r)y^i_1)
\end{align*}
where
\[
\nu_\phi(y^j_r)=n!\int_0^{\log k}e^{-ns}\lambda^1(e^{rs}y^j_r)ds.
\]
Next, let $\kappa_t$ be the 1-parameter group generated by
$\sum_{p=2}^{n+1} k_p\Delta^p_{n+1}$ for some constants $k_p$.  Now
\[
\sum_{p=2}^{n+1} k_p\Delta^p_{n+1}=
\sum_{p=2}^n (k_p\tilde{\Delta}^p) 
+\left(k_{n+1}-\frac{1}{n+1}\sum_{p=2}^nk_p\lambda^p\right)\Delta^{n+1}_{n+1}.
\]
Let $\tilde{\kappa}_t$ be the flow generated by $\sum_{p=2}^n 
(k_p\tilde{\Delta}^p)$, and $\bar{\kappa}_t$ the flow generated by 
$\sum_{p=2}^n (k_p\Delta^p_n)$, so that
\[ 
\bar{\kappa}_t\circ\tau_{n+1,n}=\tau_{n+1,n}\circ\tilde{\kappa}_t
=\tau_{n+1,n}\circ\kappa_t.
\]
For any $(y^i_r)\in T^n M$ let $\tau(t,y^i_r)$ be the solution of the
first-order ordinary differential equation
\[
\frac{d\tau}{dt}=
k_{n+1}-\frac{1}{n+1}\sum_{p=2}^nk_p\lambda^p(\bar{\kappa}_t(y^j_r))
\]
such that $\tau(0,y^i_r)=0$. We claim that 
\[
\kappa_t(\tilde{\Gamma}(y^i_r))
=\phi^{n+1}_{\tau(t,y^i_r)}(\tilde{\kappa}_t(\tilde{\Gamma}(y^i_r)),
\]
where $\phi^{n+1}_t$ is the 1-parameter group generated by
$\Delta^{n+1}_{n+1}$.  Consider the $(\dim T^n M+1)$-dimensional
submanifold $S$ of $T^{n+1}_\circ M$ consisting of points
$\phi^{n+1}_t(\tilde{\Gamma}(y^i_r))$, $t\in\R$, $(y^i_r)\in T^n_\circ
M$.  Since $\sum_{p=2}^n (k_p\tilde{\Delta}^p)$ is tangent to
$\im\tilde{\Gamma}$,
$\tilde{\kappa}_t(\tilde{\Gamma}(y^i_r))\in\im\tilde{\Gamma}$ for all
$t$, and so
$\phi^{n+1}_{\tau(t,y^i_r)}(\tilde{\kappa}_t(\tilde{\Gamma}(y^i_r))\in
S$ for all $t$.  We can use coordinates $(t,y^i_r)$ on $S$; with
respect to these coordinates $\phi^{n+1}_t$ is just translation of the
first coordinate, and
\[
\phi^{n+1}_{\tau(t,y^i_r)}(\tilde{\kappa}_t(\tilde{\Gamma}(y^i_r))
=(\tau(t,y^i_r),\bar{\kappa}_t(y^i_r)).
\]
Moreover, since $[\Delta^{n+1}_{n+1},\tilde{\Delta}^r]=0$ the
coordinate representation of $\tilde{\Delta}^r$ coincides with its
coordinate representation on $\im\tilde{\Gamma}$, which is the same as the
coordinate representation of $\Delta^r_n$.  Thus the tangent vector to
the curve
$t\mapsto\phi^{n+1}_{\tau(t,y^i_r)}(\tilde{\kappa}_t(\tilde{\Gamma}(y^i_r))$ at $t$,
in coordinate form, is
\[
\dot{\tau}(t,y^i_r)\vf{t}+\sum_{p=2}^nk_p\Delta^p_n({\bar{\kappa}}_t(y^i_r)).
\]
But $\partial/\partial t$ is the coordinate representation of
$\Delta^{n+1}_{n+1}$ on $S$, and we conclude that
$t\mapsto\phi^{n+1}_{\tau(t,y^i_r)}(\tilde{\kappa}_t(\tilde{\Gamma}(y^i_r))$ is an 
integral curve of the vector field
\[
\left(k_{n+1}-\frac{1}{n+1}\sum_{p=2}^nk_p\lambda^p\right)\Delta^{n+1}_{n+1}
+\sum_{p=2}^n (k_p\tilde{\Delta}^p)
=\sum_{p=2}^{n+1} k_p\Delta^p_{n+1}.
\]
Since 
$\phi^{n+1}_{\tau(0,y^i_r)}(\tilde{\kappa}_0(\tilde{\Gamma}(y^i_r))=\tilde{\Gamma}(y^i_r)$, 
we see that 
$t\mapsto\phi^{n+1}_{\tau(t,y^i_r)}(\tilde{\kappa}_t(\tilde{\Gamma}(y^i_r))$ is the 
integral curve of $\sum_{p=2}^{n+1} k_p\Delta^p_{n+1}$ starting at 
$\tilde{\Gamma}(y^i_r)$, that is, it coincides with $\kappa_t(\tilde{\Gamma}(y^i_r))$ 
as claimed.

It follows that
\[
\kappa_t(\tilde{\Gamma}(\bar{\kappa}_{-t}(y^i_r)))
=\phi^{n+1}_{\tau(t,\bar{\kappa}_{-t}(y^i_r))}
(\tilde{\kappa}_t(\tilde{\Gamma}(\bar{\kappa}_{-t}(y^i_r))).
\]
But $\tilde{\kappa}_t$ maps $\im\tilde{\Gamma}$ to itself, and
$\tau_{n+1,n}\circ\tilde{\kappa}_t=\bar{\kappa}_t\circ\tau_{n+1,n}$,
so
$\tilde{\kappa}_t(\tilde{\Gamma}(\bar{\kappa}_{-t}(y^i_r))=\tilde{\Gamma}(y^i_r)$.
Thus
\[
\kappa_t(\tilde{\Gamma}(\bar{\kappa}_{-t}(y^i_r)))
=\phi^{n+1}_{\tau(t,\bar{\kappa}_{-t}(y^i_r))}(\tilde{\Gamma}(y^i_r))
=(y^i_r,\Gamma^i(y^j_r)+\tau(t,\bar{\kappa}_{-t}(y^j_r))y^i_1).
\]
This establishes the required result with $\kappa_t=\phi_{n+1}^{-1}$
and $\nu_\phi(y^i_r)=\tau(t,\phi_n(y^i_r))$.  
\end{proof}

This result shows that a projective equivalence class is invariant
under the action of the group $L^{n+}$ on $T^n_\circ M$.

\section{Examples}
Take $M$ to be Euclidean space of dimension $m$ and the $y^i$ to be 
Euclidean coordinates. With the scalar product etc.\ denoted in the 
usual way, consider the third-order differential equation field 
defined by
\[
\G^i(y_1,y_2)=
-\left(\frac{2|y_1|^2|y_2|^2+(y_1\cdot y_2)^2}{2|y_1|^4}\right)y_1^i
+3\frac{(y_1\cdot y_2)}{|y_1|^2}y^i_2.
\]
This is easily seen to be homogeneous, with $\lambda_1=\lambda_2=0$.

We now explain the geometrical significance of this third-order
system.  Denote the differential equation field by $\Gamma$.  It is
easy to see that
\begin{align*}
&\Gamma(|y_1|^2)=2(y_1\cdot y_2)\\
&\Gamma(y_1\cdot y_2)=\frac{2(y_1\cdot y_2)}{|y_1|^2}.
\end{align*}
It follows that $\Gamma$ is tangent to the submanifold on which 
$|y_1|=1$, ${y_1\cdot y_2=0}$. This consists of the 2-jets of curves 
parametrized by Euclidean arc-length. The restriction of $\G^i$ to this 
submanifold is just $-|y_2|^2y^i_1$. The corresponding vector differential 
equation, considering the $y^i$ as the components of a vector $\r$ with 
respect to an orthonormal frame and using overdots to indicate 
differentiation with respect to arc-length $s$, may be written 
\[
\dddot{\r}=-|\ddot{\r}|^2\dot{\r}.
\]  
Along any solution curve of this differential equation the 
2-plane spanned by $\dot{\r}$ and $\ddot{\r}$ is constant, and so the 
curve is a plane curve. Moreover, $|\ddot{\r}|$ is the curvature of 
the curve, and 
\[
\frac{d}{ds}|\ddot{\r}|^2=2\ddot{\r}\cdot\dddot{\r}=
-2|\ddot{\r}|^2\ddot{\r}\cdot\dot{\r}=0.
\]
So the curve is a plane curve of constant curvature, that is, a circle
(or if the curvature is zero, a straight line).  

Now the submanifold $|y_1|=1$, ${y_1\cdot y_2=0}$ in this case is just
the submanifold $\S'$ of Lemma~\ref{subm}.  It follows that every
point of $T^2_\circ M$ can be written in the form
$\phi^1_{t_1}(\phi^2_{t_2}(z))$ for some $z\in\S'$, where $\phi^1_t$
and $\phi^2_t$ are the 1-parameter groups generated by $\Delta^1$ and
$\Delta^2$.  So in particular every point of $T^2_\circ M$ is the
image by some jet map of a point of $\S'$.  It follows from
Corollary~\ref{repara} that the homogeneous system we started with
also has the property that its geodesics are circles.

There will of course be other homogeneous differential equation fields
with the same property, namely those projectively equivalent to
$\Gamma$.  For example we could take the simpler system with
\[
\G^i(y_1,y_2)=3\frac{(y_1\cdot y_2)}{|y_1|^2}y^i_2. 
\] 
The corresponding field is homogeneous, now with $\lambda^2=(y_1\cdot
y_2)/|y_1|^2$ (but $\lambda^1=0$ still).  However, this field is not
tangent to $\S'$, so while it is true (by Theorem~\ref{projeq}) that
its geodesics are circles, it is not so straightforward to
see this.

We next exhibit a fourth-order system in Euclidean space derived from
a parametric second-order Lagrangian.  Consider the Lagrangian
function $L$ on $T^2_\circ M$ given (in terms of Euclidean
coordinates $(y^i)$ as before) by
\[
L(y,y_1,y_2)=\frac{|y_1|^2|y_2|^2-(y_1\cdot y_2)^2}{|y_1|^5}.
\]
This satisfies $\Delta^1(L)=L$ and $\Delta^2(L)=0$, so is parametric.
In fact $L$ is closely related to the (first) curvature of a curve,
which may be considered as a function $\kappa$ on $T^2_\circ M$:\ we
have $L=\kappa^2|y_1|$ (see for example \cite{Stoker}).  By
deriving the Euler-Lagrange equations for $L$ one obtains the
projective equivalence class of fourth-order equations, necessarily
homogeneous, containing the one given by
\[
\G^i=
-3\frac{(y_2\cdot y_3)}{|y_1|^2}y^i_1
+\left(\frac{5|y_1|^2|y_2|^2-35(y_1\cdot y_2)^2+8|y_1|^2(y_1\cdot 
y_3)}{2|y_1|^4}\right)y^i_2
+6\frac{(y_1\cdot y_2)}{|y_1|^2}y^i_3.
\]
For this particular representative $\Gamma$ we have $\lambda^1=0$, but
$\lambda^2$ is nonzero, and so of course is $\lambda^3$. We find that 
$\Gamma(|y_1|$, $\Gamma(y_1\cdot y_2)$ and $\Gamma(y_1\cdot 
y_3+|y_2|^2)$ are contained in the ring generated by these functions, 
which means that $\Gamma$ is tangent to the submanifold $\S''$ of 
$T^2_\circ M$ where $|y_1|=1$, $y_1\cdot y_2=y_1\cdot 
y_3+|y_2|^2=0$, whose points consist of jets of curves parametrized 
by arc-length. On $S''$ we obtain the differential equation 
\[
\ddddot{\r}=-3(\ddot{\r}\cdot\dddot{\r})\dot{\r}
-\threehalf|\ddot{\r}|^2\ddot{\r}
\]
(in terms of arc-length parameter). This may be more succinctly 
written as
\[
\frac{d}{ds}\left(\dddot{\r}+\threehalf|\ddot{\r}|^2\dot{\r}\right)=0.
\] 
This is the Euclidean version of an equation discussed recently by 
Matsyuk in the context of
`Zitterbewegung'~\cite{Mat} (eq (38) with $A=0$ and 
$R=0$).

\subsubsection*{Acknowledgements} The first author is a Guest
Professor at Ghent University:\ he is grateful to the Department of
Mathematics for its hospitality.  The second author acknowledges the
support of grant no.\  201/09/0981 for Global Analysis and its
Applications from the Czech Science Foundation.

\subsubsection*{Address for correspondence}
M.\ Crampin, 65 Mount Pleasant, Aspley Guise, Beds MK17~8JX, UK\\
Crampin@btinternet.com


\begin{thebibliography}{99}
\bibitem{Mat}
\arxiv{R.\,Ya.\ Matsyuk}{Higher order variational origin of the
Dixon's system and its relation to the quasi-classical
`Zitterbewegung' in General Relativity}{1104.5384}
\end{thebibliography}
\end{document}